\documentclass[12pt]{article}
\usepackage{amssymb}
\usepackage{amsthm}
\usepackage{amsmath}
\usepackage{graphicx}
\usepackage{enumerate}
\usepackage{pict2e}
\usepackage{framed}

\DeclareMathOperator{\Max}{Max}
\DeclareMathOperator{\Min}{Min}

\newtheorem{theorem}{Theorem}[section]
\newtheorem{definition}[theorem]{Definition}
\newtheorem{lemma}[theorem]{Lemma}

\newtheorem{remark}[theorem]{Remark}
\newtheorem{example}[theorem]{Example}

\title{Residuated mappings and homomorphisms in posets}
\author{Ivan~Chajda and Helmut~L\"anger}
\date{}

\begin{document}
	
\footnotetext{Support of the research of the first author by the Czech Science Foundation (GA\v CR), project 25-20013L, and support of the research of the second author by the Austrian Science Fund (FWF), project 10.55776/PIN5424624, is gratefully acknowledged.}

\maketitle
	
\begin{abstract}
The concept of weakly regular residuated mappings was investigated for lattices recently by S.~Radeleczki and L.~Veres. We modify this concept for posets. We define so-called operator-residuated mappings in posets, show their important properties and point out differences between them and residuated mappings defined in the usual way. We modify the concept of a lattice homomorphism for posets where suprema need not exist and show its relation to residuated mappings and induced ideals. Finally, we present several examples of weakly regular residuated mappings in posets and show how such mappings may be constructed.
\end{abstract}



Recently, S.~Radeleczki and L.~Veres published a paper on so-called weakly regular residuated mappings in lattices, see \cite{RV}. We recognized that some of the concepts and results contained in that paper can be modified for posets instead of lattices. Of course, not all results from \cite{RV} can be transferred to posets and, moreover, additional assumptions are required in order to obtain similar results. We believe that our approach may be inspiring for people working in posets to further developing our ideas.

\section{Preliminaries}

Some concepts used in lattice theory can be transformed to posets. For the sake of convenience of the reader we shortly summarize the fundamental concepts which we will use in the sequel.

In the following we identify singletons with their unique element.

If $M$ and $N$ are sets and $f\colon M\to N$ then we denote by $\ker f$ the equivalence relation $\{(x,y)\in M^2\mid f(x)=f(y)\}$ on $M$.

Consider a poset $\mathbf P=(P,\le)$, elements $a$ and $b$ of $P$ and subsets $A$ and $B$ of $P$. The binary relation $\le$ on $P$ can be extended to a binary relation on $2^P$ in several ways:
\begin{itemize}
\item $A\le B$ if $x\le y$ for all $x\in A$ and all $y\in B$,
\item $A\le_1B$ if for every $x\in A$ there exists some $y\in B$ with $x\le y$,
\item $A\le_2B$ if for every $y\in B$ there exists some $x\in A$ with $x\le y$.
\end{itemize}
Note that $\le_1$ and $\le_2$ are not partial order relations, but only quasiorder relations. If $\mathbf P$ has a smallest element then this element is usually denoted by $0$. Analogously, we define posets with $1$. The {\em poset} $\mathbf P$ is called {\em bounded} if it is a poset with $0$ and $1$.

Moreover, we define the down set generated by $A$ and the up set generated by $A$ as follows:
\begin{itemize}
\item $\downarrow A:=\{x\in P\mid\text{there exists some }y\in A\text{ with }x\le y\}$,
\item $\uparrow A:=\{x\in P\mid\text{there exists some }y\in A\text{ with }y\le x\}$.
\end{itemize}
The subset $I$ of $P$ is called an {\em ideal} of $\mathbf P$ if $\downarrow x\subseteq I$ and $U(x,y)\cap I\ne\emptyset$ for all $x,y\in I$.

By $(U,L)$ we denote the Galois correspondence between $(2^P,\subseteq)$ and $(2^P,\subseteq)$ induced by the binary relation $\le$ on $P$, i.e.
\begin{align*}
U(A) & :=\{y\in P\mid x\le y\text{ for all }x\in A\}, \\
L(B) & :=\{x\in P\mid x\le y\text{ for all }y\in B\}.
\end{align*}
Instead of $U(\{a\})$, $U(\{a,b\})$, $U(A\cup\{b\})$, $U(A\cup B)$, $U\big(L(A)\big)$ we simply write $U(a)$, $U(a,b)$, $U(A,b)$, $U(A,B)$ and $UL(A)$, respectively. Analogously, we proceed in similar cases.

Further $\Max A$ and $\Min A$ denote the set of all minimal and maximal elements of $A$, respectively. A subset of $P$ is a called a {\em chain} of $\mathbf P$ if any two of its elements are comparable, and it is called an {\em antichain} of $\mathbf P$ if any two different of its elements are not comparable. Let $\mathbb A(\mathbf P)$ denote the set of all non-empty antichains of $\mathbf P$. Clearly, $\Max A$ and $\Min A$ belong to $\mathbb A(\mathbf P)$ if they are not empty.

On $\mathbb A(\mathbf P)$ the quasiorder relations $\le_1$ and $\le_2$ are partial order relations. We explain it for the relation $\le_1$. Assume $A,B\in\mathbb A(\mathbf P)$, $A\le_1B$ and $B\le_1A$. Suppose $a\in A$. Since $A\le_1B$ there exists some $b\in B$ with $a\le b$. Since $B\le_1A$ there exists some $c\in A$ with $b\le_1c$. Since $A\in\mathbb A(\mathbf P)$, $a,c\in A$ and $a\le c$ we have $a=c$ and hence $a=b\in B$. This shows $A\subseteq B$. By symmetry we obtain $B\subseteq A$ and hence $A=B$.

The poset $\mathbf P$ is said to satisfy the {\em Ascending Chain Condition} if it has no infinite ascending chains, and $\mathbf P$ is said to satisfy the {\em Descending Chain Condition} if it has no infinite descending chains. Of course, every finite poset satisfies both chain conditions.

Let $\mathbf Q=(Q,\le)$ be a further poset, $f\colon P\to Q$ and $g\colon P\to P$. Then
\begin{itemize}
\item $f$ is called {\em monotone} if for all $x,y\in P$, $x\le y$ implies $f(x)\le f(y)$,
\item $g$ is called {\em intensive} if $g(x)\le x$ for all $x\in P$,
\item $g$ is called {\em extensive} if $x\le g(x)$ for all $x\in P$.
\end{itemize}
The mapping $f$ is called an {\em order isomorphism} if it is a bijection from $P$ to $Q$ and if for all $x,y\in P$ we have $x\le y$ if and only if $f(x)\le f(y)$. The mapping $f$ is called an {\em up-homomorphism} from $\mathbf P$ to $\mathbf Q$ if
\[
f\big(\Min U(x,y)\big)=\Min U\big(f(x),f(y)\big)\tag{1}
\]
for all $x,y\in P$. By an {\em operator} on $\mathbf P$ we mean a mapping from $P$ to $2^Q$.

\section{Operator-residuated mappings}

When considering residuated mappings in posets we distinguish between the following two concepts.

\begin{definition}
Let $\mathbf P=(P,\le)$ and $\mathbf Q=(Q,\le)$ be posets. A mapping $f\colon P\to Q$ is called an {\em operator-residuated} mapping from $\mathbf P$ to $\mathbf Q$ if there exists some operator $f^*\colon Q\to\mathbb A(\mathbf P)$ such that for all $x\in P$ and all $y\in Q$, 
\[
f(x)\le y\text{ is equivalent to }x\le_1f^*(y).\tag{2}
\]
In this case the couple $(f,f^*)$ is called a {\em pair of operator-residuated mappings} between $\mathbf P$ and $\mathbf Q$. If $|f^*(y)|=1$ for all $y\in Q$ then the operator $f^*$ on $\mathbf Q$ is a mapping from $Q$ to $P$ and in this case, if {\rm(2)} is satisfied for all $x\in P$ and all $y\in Q$, $f$ is called a {\em residuated mapping} from $\mathbf P$ to $\mathbf Q$ and the couple $(f,f^*)$ a {\em pair of residuated mappings} between $\mathbf P$ and $\mathbf Q$.
\end{definition}

It turns out that in a pair of operator-residuated mappings, every of its two components uniquely determines the other one.

\begin{lemma}\label{lem3}
Let $\mathbf P=(P,\le)$ and $\mathbf Q=(Q,\le)$ be posets and $(f,f^*)$ a pair of operator-residuated mappings between $\mathbf P$ and $\mathbf Q$. Then
\begin{enumerate}[{\rm(i)}]
\item $f$ uniquely determines by $f^*$,
\item $f^*$ uniquely determines by $f$.
\end{enumerate}
\end{lemma}

\begin{proof}
$\text{}$	
\begin{enumerate}[(i)]
\item Let $d\in Q$ and $A,B\in\mathbb A(\mathbf P)$ and assume that
\[
f(x)\le d\text{ if and only if }x\le_1A\text{ if and only if }x\le_1B,
\]
for all $x\in P$. Assume $a\in A$. Then $a\le_1A$ and hence $a\le_1B$ which implies that there exists some $b\in B$ with $a\le b$. Now $b\le_1B$ and therefore $b\le_1A$ which shows that there exists some $c\in A$ with $b\le c$. Together we obtain $a\le b\le c$. Since $A$ is an antichain we have $a=c$ and therefore $a=b=c$ showing $a=b\in B$. Since $a$ was an arbitrary element of $A$, we have proved $A\subseteq B$. By symmetry, we also have $B\subseteq A$ and hence $A=B$ completing the proof of (i).
\item Let $x\in P$. Since for all $y\in Q$, $y\ge f(x)$ is equivalent to $y\in\{z\in Q\mid x\le_1f^*(z)\}$, we have that $f(x)$ is the smallest element of $\{z\in Q\mid x\le_1f^*(z)\}$.
\end{enumerate}
\end{proof}

\begin{example}
Consider the poset $\mathbf P$ depicted in Fig.~1:	
	
\vspace*{-3mm}
	
\begin{center}
\setlength{\unitlength}{7mm}
\begin{picture}(6,10)
\put(3,1){\circle*{.3}}
\put(1,3){\circle*{.3}}
\put(5,3){\circle*{.3}}
\put(1,5){\circle*{.3}}
\put(5,5){\circle*{.3}}
\put(1,7){\circle*{.3}}
\put(5,7){\circle*{.3}}
\put(3,9){\circle*{.3}}
\put(3,1){\line(-1,1)2}
\put(3,1){\line(1,1)2}
\put(3,9){\line(-1,-1)2}
\put(3,9){\line(1,-1)2}
\put(1,3){\line(0,1)4}
\put(5,3){\line(0,1)4}
\put(1,3){\line(2,1)4}
\put(1,5){\line(2,1)4}
\put(5,3){\line(-2,1)4}
\put(5,5){\line(-2,1)4}
\put(2.85,.3){$0$}
\put(.35,2.85){$a$}
\put(5.4,2.85){$b$}
\put(.35,4.85){$c$}
\put(5.4,4.85){$d$}
\put(.35,6.85){$e$}
\put(5.4,6.85){$g$}
\put(2.85,9.4){$1$}
\put(-.15,-.75){{\rm Figure~1. Finite poset $\mathbf P$}}
\end{picture}
\end{center}
	
\vspace*{4mm}
	
Then $(f,f^*)$ defined by
\[
\begin{array}{r|c|c|c|c|c|c|c|c}
     x & 0 & a & b &    c    & d &    e   &    g   & 1 \\
\hline     
  f(x) & 0 & c & c &    g    & g &    1   &    1   & 1 \\
\hline  
f^*(x) & 0 & 0 & 0 & \{a,b\} & 0 & \{a,b\}& \{c,d\}& 1
\end{array}
\]
is a pair of operator-residuated mappings between $\mathbf P$ and itself.
\end{example}

The following result is a generalization of a well-known result on Galois correspondences. Galois correspondences induced by a binary relation were studied in \cite B under the name polarity.

\begin{theorem}\label{th1}
Let $\mathbf P=(P,\le)$ and $\mathbf Q=(Q,\le)$ be posets, $f\colon P\to Q$ and $g\colon Q\to\mathbb A(\mathbf P)$. Then the following are equivalent:
\begin{enumerate}[{\rm(i)}]
\item $(f,f^*)$ is a pair of operator-residuated mappings between $\mathbf P$ and $\mathbf Q$,
\item For every $x\in P$, $f(x)$ is the smallest element of $\{z\in Q\mid x\le_1f^*(z)\}$,
\item $f$ is monotone, $f^*(x)\le_1f^*(y)$ for all $x,y\in Q$ with $x\le y$, $x\le_1f^*\big(f(x)\big)$ for all $x\in P$ and $f\big(f^*(y)\big)\le y$ for all $y\in Q$.	
\end{enumerate}
\end{theorem}

\begin{proof}
$\text{}$ \\
(i) $\Leftrightarrow$ (ii): \\
This follows from the fact that for every $x\in P$ and $y\in Q$, $x\le_1f^*(y)$ is equivalent to $y\in\{z\in Q\mid x\le_1f^*(z)\}$. \\
(i) $\Rightarrow$ (iii): \\
Let $a,b\in P$ and $c,d\in Q$. Because of $f(a)\le f(a)$ we have $a\le_1f^*\big(f(a)\big)$ and because of $x\le_1f^*(c)$ for all $x\in f^*(c)$ we have $f(x)\le c$ for all $x\in f^*(c)$, i.e.\ $f\big(f^*(c)\big)\le c$. If $a\le b$ then $a\le b\le_1f^*\big(f(b)\big)$ and hence $a\le_1f^*\big(f(b)\big)$ which implies $f(a)\le f(b)$ showing monotonicity of $f$. If, finally, $c\le d$ then $f\big(f^*(c)\big)\le c\le d$, i.e.\ $f(x)\le d$ for all $x\in f^*(c)$ and hence $x\le_1f^*(d)$ for all $x\in f^*(c)$, i.e.\ $f^*(c)\le_1f^*(d)$. \\
(iii) $\Rightarrow$ (i): \\
$f(a)\le c$ implies $a\le_1f^*\big(f(a)\big)\le_1f^*(c)$ and therefore $a\le_1f^*(c)$. Conversely, if $a\le_1f^*(c)$ then there exists some $e\in f^*(c)$ with $a\le e$ and hence $f(a)\le f(e)\in f\big(f^*(c)\big)\le c$ whence $f(a)\le c$.
\end{proof}

For posets with $0$ or $1$ we obtain

\begin{lemma}\label{lem4}
Let $\mathbf P=(P,\le)$ and $\mathbf Q=(Q,\le)$ be posets and $(f,f^*)$ a pair of operator-residuated mappings between $\mathbf P$ and $\mathbf Q$. Then the following holds:
\begin{enumerate}[{\rm(i)}]
\item If $\mathbf P$ and $\mathbf Q$ are posets with $0$ then
\begin{enumerate}[{\rm(a)}]
\item $f(0)=0$,
\item $f^{-1}(0)=[0](\ker f)=\downarrow\big(f^*(0)\big)$,
\item if $f^{-1}(0)$ has a greatest element then $|f^*(0)|=1$ and $f^{-1}(0)=[0,f^*(0)]$,
\end{enumerate}
\item if $\mathbf P$ and $\mathbf Q$ are posets with $1$ then $f^*(1)=1$.
\end{enumerate}
\end{lemma}

\begin{proof}
$\text{}$ \\
\begin{enumerate}[(i)]
\item Assume $\mathbf P$ and $\mathbf Q$ to be posets with $0$ and let $a\in P$.
\begin{enumerate}[(a)]
\item The following are equivalent: $0\le_1f^*(0)$, $f(0)\le0$, $f(0)=0$.
\item Because of (a) the following are equivalent: $a\in f^{-1}(0)$, $f(a)=0$, $f(a)=f(0)$, $a\in [0](\ker f)$. Moreover, the following are equivalent: $a\in f^{-1}(0)$, $f(a)=0$, $f(a)\le0$, $a\le_1f^*(0)$, $a\in\downarrow\big(f^*(0)\big)$.
\item If $f^{-1}(0)$ has a greatest element then because of (b) we have $|f^*(0)|=1$ and $f^{-1}(0)=[0,f^*(0)]$
\end{enumerate}
\item If $\mathbf P$ and $\mathbf Q$ are posets with $1$ then, since $f^*(1)$ is an antichain of $\mathbf P$, the following are equivalent: $f(1)\le1$, $1\le_1f^*(1)$, $1\in f^*(1)$, $f^*(1)=1$.
\end{enumerate}
\end{proof}

Let $A$ and $B$ be sets and $f\colon A\to2^B$. We extend $f$ from $A$ to $2^A$ by defining $f(C):=\bigcup\limits_{x\in C}f(x)$ for all $C\subseteq A$. If $A_i\subseteq A$ for all $i\in I$ then
\[
f\left(\bigcup_{i\in I}A_i\right)=\bigcup_{x\in\bigcup\limits_{i\in I}A_i}f(x)=\bigcup_{i\in I}\bigcup_{x\in A_i}f(x)=\bigcup_{i\in I}f(A_i).
\]

The next theorem shows how the mapping $f$ and the operator $f^*$ are related with the compositions $f\circ f^*\circ f$ and $f^*\circ f\circ f^*$.

\begin{theorem}\label{th2}
Let $\mathbf P=(P,\le)$ and $\mathbf Q=(Q,\le)$ be posets, $(f,f^*)$ a pair of operator-residuated mappings between $\mathbf P$ and $\mathbf Q$, $a\in P$ and $b\in Q$. Then
\begin{enumerate}[{\rm(i)}]
\item $f(a)\le_1f\Big(f^*\big(f(a)\big)\Big)\le f(a)$,
\item $f^*(b)\le_1f^*\Big(f\big(f^*(b)\big)\Big)\le_1f^*(b)$.
\end{enumerate}
\end{theorem}

\begin{proof}
In this proof we often use Theorem~\ref{th1}.
\begin{enumerate}[(i)]
\item Because of $a\le_1f^*\big(f(a)\big)$ there exists some $c\in f^*\big(f(a)\big)$ with $a\le c$. Now $f(a)\le f(c)\in f\Big(f^*\big(f(a)\big)\Big)$. This shows $f(a)\le_1f\Big(f^*\big(f(a)\big)\Big)$. On the other hand, we have $f\big(f^*(y)\big)\le y$ for all $y\in Q$ and hence $f\Big(f^*\big(f(a)\big)\Big)\le f(a)$.
\item We have $x\le_1f^*\big(f(x)\big)$ for all $x\in P$ and therefore
\[
f^*(b)=\bigcup_{x\in f^*(b)}x\le_1\bigcup_{x\in f^*(b)}f^*\big(f(x)\big)=f^*(f\left(\bigcup_{x\in f^*(b)}x\right))=f^*\Big(f\big(f^*(b)\big)\Big).
\]
On the other hand, $f\big(f^*(b)\big)\le b$, i.e.\ $f(x)\le b$ for all $x\in f^*(b)$. This implies $f^*\big(f(x)\big)\le_1f^*(b)$ for all $x\in f^*(b)$. Now
\[
f^*\Big(f\big(f^*(b)\big)\Big)=f^*(f\left(\bigcup_{x\in f^*(b)}x\right))=\bigcup_{x\in f^*(b)}f^*\big(f(x)\big)\le_1f^*(b).
\]
\end{enumerate}	
\end{proof}

For posets satisfying the Ascending Chain Condition we can describe the operator $f^*$ as follows.

\begin{theorem}\label{th4}
Let $\mathbf P=(P,\le,0)$ and $\mathbf Q=(Q,\le,0)$ be posets with $0$, assume $\mathbf P$ to satisfy the Ascending Chain Condition and assume $f\colon P\to Q$ to be monotone. Define $f^*\colon Q\to\mathbb A(\mathbf P)$ by
\[
f^*(y):=\Max\{x\in P\mid f(x)\le y\}\text{ for all }y\in Q.\tag{2}
\]
Then $(f,f^*)$ is a pair of operator-residuated mappings between $\mathbf P$ and $\mathbf Q$.
\end{theorem}

\begin{proof}
Let $a\in P$ and $b\in Q$. Since $f(0)=0\le b$ according to Lemma~\ref{lem4} and $\mathbf P$ satisfies the Ascending Chain Condition, $f^*(b)\in\mathbb A(\mathbf P)$. Now assume $f(a)\le b$. Then $a\in\{x\in P\mid f(x)\le b\}$. Since $\mathbf P$ satisfies the Ascending Chain Condition, $a\le_1f^*(b)$. Conversely, assume $a\le_1f^*(b)$. Then there exists some $c\in f^*(b)$ with $a\le c$. Now $c\in\{x\in P\mid f(x)\le b\}$ and hence $f(a)\le f(c)\le b$ showing $f(a)\le b$.
\end{proof}

\section{Pairs of residuated mappings}

In this section we consider pairs $(f,f^*)$ of residuated mappings between two posets $(P,\le)$ and $(Q,\le)$, i.e.\ we assume $|f^*(y)|=1$ for all $y\in Q$.

\begin{example}\label{ex1}
Consider the poset $\mathbf P$ visualized in Fig.~2:	
	
\vspace*{-3mm}
	
\begin{center}
\setlength{\unitlength}{7mm}
\begin{picture}(6,8)
\put(3,1){\circle*{.3}}
\put(1,3){\circle*{.3}}
\put(5,3){\circle*{.3}}
\put(1,5){\circle*{.3}}
\put(5,5){\circle*{.3}}
\put(3,7){\circle*{.3}}
\put(3,1){\line(-1,1)2}
\put(3,1){\line(1,1)2}
\put(3,7){\line(-1,-1)2}
\put(3,7){\line(1,-1)2}
\put(1,3){\line(0,1)2}
\put(5,3){\line(0,1)2}
\put(1,3){\line(2,1)4}
\put(5,3){\line(-2,1)4}
\put(2.85,.3){$0$}
\put(.35,2.85){$a$}
\put(5.4,2.85){$b$}
\put(.35,4.85){$c$}
\put(5.4,4.85){$d$}
\put(2.85,7.4){$1$}
\put(-.15,-.75){{\rm Figure~2. Finite poset $\mathbf P$}}
\end{picture}
\end{center}
	
\vspace*{4mm}
	
Then $(f,f^*)$ where $f$ and $f^*$ are defined by
\[
\begin{array}{r|c|c|c|c|c|c}
     x & 0 & a & b & c & d & 1 \\
\hline     
  f(x) & 0 & 0 & c & c & c & c \\
\hline  
f^*(x) & a & a & a & 1 & a & 1
\end{array}
\]
is a pair of residuated mappings between $\mathbf P$ and itself.
\end{example}

For residuated mappings the equivalence of the statements (i), (iii) and (iv) of the following theorem follow directly from Theorem~\ref{th1}.

\begin{theorem}\label{th3}
Let $\mathbf P=(P,\le)$ and $\mathbf Q=(Q,\le)$ be posets, $f\colon P\to Q$ and $f^*\colon Q\to P$. Then the following are equivalent:
\begin{enumerate}[{\rm(i)}]
\item $(f,f^*)$ is a pair of residuated mappings between $\mathbf P$ and $\mathbf Q$,
\item for all $y\in Q$, $f^*(y)$ is the greatest element of $\{z\in P\mid f(z)\le y\}$, 
\item for all $x\in P$, $f(x)$ is the smallest element of $\{z\in Q\mid f^*(z)\ge x\}$,
\item $f$ and $f^*$ are monotone, $f\circ f^*$ is intensive and $f^*\circ f$ is extensive.
\end{enumerate}	
\end{theorem}

\begin{proof}
The equivalence of (i) and (ii) follows from the fact that for $x\in P$ and $y\in Q$, $f(x)\le y$ is equivalent to $x\in\{z\in P\mid f(z)\le y\}$.
\end{proof}

In the case of residuated mappings Lemma~\ref{lem4} reads as follows:

\begin{lemma}\label{lem2}
Let $\mathbf P=(P,\le)$ and $\mathbf Q=(Q,\le)$ be posets and $(f,f^*)$ a pair of residuated mappings between $\mathbf P$ and $\mathbf Q$. If $\mathbf P$ and $\mathbf Q$ are posets with $0$ then $f(0)=0$, and if $\mathbf P$ and $\mathbf Q$ are posets with $1$ then $f^*(1)=1$.
\end{lemma}

If $f$ is a residuated mapping then the result of Theorem~\ref{th2} can be stated in a stronger version as follows.

\begin{theorem}
If $\mathbf P$ and $\mathbf Q$ are posets and $(f,f^*)$ is a pair of residuated mappings between $\mathbf P$ and $\mathbf Q$. Then $f\circ f^*\circ f=f$ and $f^*\circ f\circ f^*=f^*$.
\end{theorem}

\begin{remark}
An interesting example of a pair of residuated mappings between a poset $\mathbf P=(P,\le)$ and itself is the pair $(f,f)$ where $f$ is a monotone mapping from $\mathbf P$ to $\mathbf P$ satisfying $f\big(f(x)\big)=x$ for all $x\in P$. Namely, if $a,b\in P$ then $f(a)\le b$ implies $a=f\big(f(a)\big)\le f(b)$ and, conversely, $a\le f(b)$ implies $f(a)\le f\big(f(b)\big)=b$.
\end{remark} 	

\section{up-homomorphisms}

In this section we show that up-homomorphisms are residuated mappings provided the corresponding posets satisfy the Ascending Chain Condition. Hence one might expect that these homomorphisms satisfy some stronger condition.

\begin{lemma}\label{lem1}
Let	$\mathbf P=(P,\le,0)$ and $\mathbf Q=(Q,\le,0)$ be posets with $0$ and let $f\colon P\to Q$. Then the following holds:
\begin{enumerate}[{\rm(i)}]
\item If $f$ is an up-homomorphism from $\mathbf P$ to $\mathbf Q$ then $f$ is monotone,
\item if $f$ is an up-homomorphism from $\mathbf P$ to $\mathbf Q$ and $\mathbf P$ satisfies the Ascending Chain Condition then $f$ is operator-residuated,
\item if $f$ is monotone then $f$ is an up-homomorphism from $\mathbf P$ to $\mathbf Q$ if and only if {\rm(1)} holds for all $x,y\in P$ with $x\parallel y$.
\end{enumerate}
\end{lemma}

\begin{proof}
Let $a,b\in P$.
\begin{enumerate}[(i)]
\item If $a\le b$ then
\[
f(b)=f\big(\Min U(b)\big)=f\big(\Min U(a,b)\big)=\Min U\big(f(a),f(b)\big)\subseteq U\big(f(a),f(b)\big)
\]
showing $f(a)\le f(b)$.
\item This follows from (i) and from Theorem~\ref{th4}.
\item If $f$ is monotone then (1) holds for all $x,y\in P$ with $x\not\parallel y$ since $a\le b$ implies
\[
f\big(\Min U(a,b)\big)=f\big(\Min U(b)\big)=f(b)=\Min U\big(f(b)\big)=\Min U\big(f(a),f(b)\big).
\]
\end{enumerate}		
\end{proof}

\begin{example}
Consider the poset $\mathbf P$ and the mappings $f$ and $f^*$ from Example~\ref{ex1}. Since $(f,f^*)$ is a pair of residuated mappings between $\mathbf P$ and itself, $f$ is monotone according to Theorem~\ref{th3}. Since
\begin{align*}
f\big(\Min U(a,b)\big) & =f(\{c,d\})=c=\Min U(c)=\Min U(0,c)=\Min U\big(f(a),f(b)\big), \\
f\big(\Min U(c,d)\big) & =f(1)=c=\Min U(c)=\Min U\big(f(c),f(d)\big),
\end{align*}
$f$ is an up-homomorphism from $\mathbf P$ to $\mathbf P$ according to Lemma~\ref{lem1}.
\end{example}

It is well known that homomorphisms between lattices induce ideals. A similar situation occurs with up-homomorphisms between posets.

\begin{theorem}
Let $\mathbf P=(P,\le,0,1)$ be a bounded poset satisfying the Descending Chain Condition, $\mathbf Q=(Q,\le,0)$ a poset with $0$ and $f$ an up-homomorphism from $\mathbf P$ to $\mathbf Q$. Then the following holds:
\begin{enumerate}[{\rm(i)}]
\item $[0](\ker f)$ is an ideal of $\mathbf P$,
\item if $\mathbf P$ satisfies the Ascending Chain Condition then $|f^*(0)|=1$ and $I=[0,f^*(0)]$.
\end{enumerate}
\end{theorem}

\begin{proof}
Put $I:=[0](\ker f)$. \\
\begin{enumerate}[(i)]
\item If $a\in P$, $b\in I$ and $a\le b$ then, because $f$ is monotone, we have $f(0)\le f(a)\le f(b)=f(0)$ and hence $f(a)=f(0)$, i.e.\ $a\in I$. Now assume $a,b\in I$. Since $f$ is an up-homomorphism from $\mathbf P$ to $\mathbf Q$ we have
\[
f\big(\Min U(a,b)\big)=\Min U\big(f(a),f(b)\big)=\Min U\big(f(0)\big)=f(0).
\]
Because of $1\in U(a,b)$ we have $U(a,b)\ne\emptyset$. Since $\mathbf P$ satisfies the Descending Chain Condition we conclude $\Min U(a,b)\ne\emptyset$. Let $c\in\Min U(a,b)$. Then $f(c)\in f\big(\Min U(a,b)\big)=f(0)$, i.e.\ $c\in I$. This shows $U(a,b)\cap I\ne\emptyset$ completing the proof of (i).
\item Put $A:=\Max I$. Since $\mathbf P$ satisfies the Ascending Chain condition we have $A\ne\emptyset$. Now assume $|A|>1$. Let $a,b\in A$ with $a\ne b$. Because of the definition of $A$, $a\parallel b$. From (i) we conclude that there exists some $d\in U(a,b)\cap I$. Since $a\parallel b$ we have $d>a,b$ contradicting the maximality of $a$ and $b$. This shows $|A|=1$, say $A=e$. Now let $g\in I$. Since $I$ is an ideal of $\mathbf P$ there exists some $h\in U(g,e)\cap I$. Now $h>e$ is impossible because of the maximality of $e$. This shows $h=e$ and therefore $g\le h=e$ showing that $e$ is the greatest element of $I$. According to Lemma~\ref{lem4}, $e=f^*(0)$.
\end{enumerate}
\end{proof}

The following theorem is some kind of generalization of (ii) of Lemma~1 in \cite{RV}.

\begin{theorem}
Let $\mathbf P=(P,\le,0,1)$ be a bounded poset satisfying the Ascending Chain Condition as well as the Descending Chain Condition, $\mathbf Q=(Q,\le,0)$ a poset with $0$, $f$ an up-homomorphism from $\mathbf P$ to $\mathbf Q$ and $a,b\in P$. Then
\[
(a,b)\in\ker f\text{ if and only if }\Min U\Big(f^*\big(f(a)\big),f^*(0)\Big)=\Min U\Big(f^*\big(f(b)\big),f^*(0)\Big).
\]
\end{theorem}

\begin{proof}
According to Lemma~\ref{lem1}, $f$ is a residuated mapping. Because $\mathbf P$ is a poset with $1$ and satisfies the Descending Chain Condition we have
\[
\Min U\Big(f^*\big(f(a)\big),f^*(0)f\Big)\ne\emptyset.
\]
Since $f$ is an up-homomorphism from $\mathbf P$ to $\mathbf Q$ we have
\[
f\bigg(\Min U\Big(f^*\big(f(a)\big),f^*(0)\Big)\bigg)=\Min U\bigg(f\Big(f^*\big(f(a)\big)\Big),f\big(f^*(0)\big)\bigg)=\Min\big(f(a),0\big)=f(a)
\]
and hence $\Min U\Big(f^*\big(f(a)\big),f^*(0)\Big)\le f^*\big(f(a)\big)$. On the other hand,
\[
\Min U\Big(f^*\big(f(a)\big),f^*(0)\Big)\ge f^*\big(f(a)\big)
\]
showing $\Min U\Big(f^*\big(f(a)\big),f^*(0)\Big)=f^*\big(f(a)\big)$. The result now follows from the fact that the following are equivalent: $(a,b)\in\ker f$, $f(a)=f(b)$ and $f^*\big(f(a)\big)=f^*\big(f(b)\big)$.
\end{proof}

\section{Weakly regular residuated mappings in posets}

\begin{lemma}\label{lem5}
Let $\mathbf P=(P,\le,1)$ be a poset with $1$, $\mathbf Q=(Q,\le,0)$ a poset with $0$ and $(f,f^*)$ a pair of operator-residuated mappings between $\mathbf P$ and $\mathbf Q$. Then the following holds:
\begin{enumerate}[{\rm(i)}]
\item $f^{-1}(0)=\downarrow\big(f^*(0)\big)$,
\item $(f^*)^{-1}(1)=\uparrow\big(f(1)\big)$.
\end{enumerate}	
\end{lemma}

\begin{proof}
Let $a\in P$ and $b\in Q$.
\begin{enumerate}[(i)]
\item The following are equivalent: $a\in f^{-1}(0)$, $f(a)=0$, $f(a)\le0$, $a\le_1f^*(0)$, \\ $a\in\downarrow\big(f^*(0)\big)$.
\item The following are equivalent: $b\in(f^*)^{-1}(1)$, $f^*(b)=1$, $1\le_1f^*(b)$, $f(1)\le b$, $b\in\uparrow\big(f(1)\big)$. (Observe that $f^*(b)$ is an antichain.)
\end{enumerate}
\end{proof}

We now adopt the following concepts from \cite{RV}.

\begin{definition}
Let $\mathbf P=(P,\le,1)$ be a poset with $1$, $\mathbf Q=(Q,\le,0)$ a poset with $0$ and $f$ a residuated mapping from $\mathbf P$ to $\mathbf Q$. Then $f$ is called
\begin{itemize}
\item {\em range-closed} if $f(P)=[0,f(1)]$,
\item {\em dually range-closed} if $f^*(Q)=[f^*(0),1]$,
\item {\em weakly regular} if it is both range-closed and dually range-closed.
\end{itemize}	
\end{definition}

The next result shows how weakly regular mappings can be constructed in certain posets having atoms and coatoms.

\begin{theorem}
Let $\mathbf P=(P,\le,1)$ be a poset with $1$, $\mathbf Q=(Q,\le,0)$ a poset with $0$, $a$ an atom of $\mathbf Q$ and $b\in P$. If there exists some coatom $c$ of $\mathbf P$ with $c\not\ge b$ and $(\downarrow c)\cup[b,1]=P$ and some $d\in Q$ with $d\not\ge a$ and $[0,d]\cup(\uparrow a)=Q$ then the mapping $f\colon P\to Q$ defined by
\[
f(x):=\left\{
\begin{array}{ll}
a & \text{if }x\ge b \\
0 & \text{otherwise}
\end{array}
\right.
\]
is a weakly regular residuated mapping from $\mathbf P$ to $\mathbf Q$.
\end{theorem}

\begin{proof}
Define $f^*\colon Q\to P$ by
\[
f^*(y):=\left\{
\begin{array}{ll}
c & \text{if }y\le d \\
1 & \text{otherwise}
\end{array}
\right.
\]
We show that $(f,f^*)$ is a pair of residuated mappings between $\mathbf P$ and $\mathbf Q$. Let $e\in P$ and $g\in Q$. \\
Case 1. $e\ge b$ and $g\le d$. \\
Because of $a\not\le d$ we have $e\not\le c=f^*(g)$. \\
Case 2. $e\ge b$ and $g\not\le d$. \\
Because of $g\not\le d$ and $Q=[0,d]\cup(\uparrow a)$ we have $f(e)=a\le g$. Moreover, $e\le1=f^*(g)$. \\
Case 3. $e\not\ge b$ and $g\le d$. \\
Then $f(e)=0\le g$. Because of $e\not\ge b$ and $P=(\downarrow c)\cup[b,1]$ we have $e\le c=f^*(g)$. \\
Case 4. $e\not\ge b$ and $g\not\le d$. \\
We have $f(e)=0\le g$ and $e\le1=f^*(g)$. \\
Hence in any case $f(e)\le g$ if and only if $e\le f^*(g)$. This shows that $(f,f^*)$ is a pair of residuated mappings between $\mathbf P$ and $\mathbf Q$. Observe $f(0)=0$ and $f^*(a)=1$. Finally, $f(P)=\{0,a\}=[0,a]=[0,f(1)]$ and $f^*(Q)=\{c,1\}=[c,1]=[f^*(0),1]$.
\end{proof}

\begin{example}
Consider the poset $\mathbf P$ depicted in Fig.~3:	
	
\vspace*{-3mm}
	
\begin{center}
\setlength{\unitlength}{7mm}
\begin{picture}(6,8)
\put(3,1){\circle*{.3}}
\put(1,3){\circle*{.3}}
\put(3,3){\circle*{.3}}
\put(5,3){\circle*{.3}}
\put(1,5){\circle*{.3}}
\put(3,5){\circle*{.3}}
\put(5,5){\circle*{.3}}
\put(3,7){\circle*{.3}}
\put(3,1){\line(-1,1)2}
\put(3,1){\line(0,1)6}
\put(3,1){\line(1,1)2}
\put(3,7){\line(-1,-1)2}
\put(3,7){\line(1,-1)2}
\put(1,3){\line(0,1)2}
\put(1,3){\line(1,1)2}
\put(3,3){\line(-1,1)2}
\put(3,3){\line(1,1)2}
\put(5,3){\line(-1,1)2}
\put(5,3){\line(0,1)2}
\put(2.85,.3){$0$}
\put(.35,2.85){$a$}
\put(3.4,2.85){$e$}
\put(5.4,2.85){$b$}
\put(.35,4.85){$c$}
\put(3.4,4.85){$g$}
\put(5.4,4.85){$d$}
\put(2.85,7.4){$1$}
\put(-.15,-.75){{\rm Figure~3. Finite poset $\mathbf P$}}
\end{picture}
\end{center}
	
\vspace*{4mm}
	
Then $(f,f^*)$ where $f$ and $f^*$ are defined by
\[
\begin{array}{r|c|c|c|c|c|c|c|c}
     x & 0 & a & e & b & c & g & d & 1 \\
\hline     
  f(x) & 0 & 0 & 0 & a & 0 & a & a & a\\
\hline  
f^*(x) & c & 1 & c & c & 1 & 1 & c & 1
\end{array}
\]
is a pair of residuated mappings between $\mathbf P$ and itself.
\end{example}

Another method for constructing weakly regular mappings by using an order isomorphism is as follows.

\begin{theorem}
Let $\mathbf P=(P,\le,1)$ be a poset with $1$, $\mathbf Q=(Q,\le,0)$ a poset with $0$, $a\in P$, $b\in Q$ and $g$ an order-isomorphism from $([a,1],\le)$ to $([0,b],\le)$, assume $P=(\downarrow a)\cup[a,1]$ and $Q=[0,b]\cup(\uparrow b)$ and define $f\colon P\to Q$ by
\[
f(x):=\left\{
\begin{array}{ll}
0    & \text{if }x\le a \\
g(x) & \text{otherwise}
\end{array}
\right.
\]
Then $f$ is a weakly regular residuated mapping from $\mathbf P$ to $\mathbf Q$.
\end{theorem}

\begin{proof}
Define $f^*\colon Q\to P$ by
\[
f^*(y):=\left\{
\begin{array}{ll}
g^{-1}(y) & \text{if }y\le b \\
1         & \text{otherwise}
\end{array}
\right.
\]
Let $c\in P$ and $d\in Q$. \\
Case 1. $c\le a$ and $d\le b$. \\
Then $f(c)=0\le d$ and $c\le a\le g^{-1}(d)=f^*(d)$. \\
Case 2. $c\le a$ and $d\not\le b$. \\
Then $f(c)=0\le d$ and $c\le1=f^*(d)$. \\
Case 3. $c\not\le a$ and $d\le b$. \\
Because of $c\not\le a$ and $P=(\downarrow a)\cup[a,1]$ we have $c\ge a$. Now the following are equivalent: $f(c)\le d$, $g(c)\le d$, $c\le g^{-1}(d)$, $c\le f^*(d)$. \\
Case 4. $c\not\le a$ and $d\not\le b$. \\
Because of $c\not\le a$ and $P=(\downarrow a)\cup[a,1]$ we have $c\ge a$. Because of $d\not\le b$ and $Q=[0,b]\cup(\uparrow b)$ we have $d\ge b$. Now $f(c)=g(c)\le b\le d$ and $c\le1=f^*(d)$. \\
Hence in any case $f(c)\le d$ if and only if $c\le f^*(d)$. This shows that $(f,f^*)$ is a pair of residuated mappings between $\mathbf P$ and $\mathbf Q$. Finally, $f(P)=[0,b]=[0,f(1)]$ and $f^*(Q)=[a,1]=[f^*(0),1]$.
\end{proof}

\begin{example}\label{ex1}
Consider the posets $\mathbf P$ and $\mathbf Q$ visualized in Fig.~4:	
	
\vspace*{4mm}
	
\begin{center}
\setlength{\unitlength}{7mm}
\begin{picture}(11,6)
\put(0,1){\circle*{.3}}
\put(4,1){\circle*{.3}}
\put(9,1){\circle*{.3}}
\put(2,3){\circle*{.3}}
\put(9,3){\circle*{.3}}
\put(2,5){\circle*{.3}}
\put(7,5){\circle*{.3}}
\put(11,5){\circle*{.3}}
\put(2,3){\line(-1,-1)2}
\put(2,3){\line(1,-1)2}
\put(2,3){\line(0,1)2}
\put(9,3){\line(0,-1)2}
\put(9,3){\line(-1,1)2}
\put(9,3){\line(1,1)2}
\put(-.15,.3){$c$}
\put(3.85,.3){$d$}
\put(8.85,.3){$0$}
\put(2.4,2.85){$a$}
\put(9.4,2.85){$b$}
\put(1.85,5.4){$1$}
\put(6.85,5.4){$e$}
\put(10.85,5.4){$g$}
\put(1.7,6.4){$\mathbf P$}
\put(8.7,6.4){$\mathbf Q$}
\put(2.5,-.75){{\rm Figure~4. Finite posets}}
\end{picture}
\end{center}
	
\vspace*{4mm}

Then $(f,f^*)$ where $f$ and $f^*$ are defined by
\[
\begin{array}{r|c|c|c|c}
   x & c & d & a & 1 \\
\hline     
f(x) & 0 & 0 & 0 & b
\end{array}
\quad\quad\quad
\begin{array}{r|c|c|c|c}
     y & 0 & b & e & g \\
\hline     
f^*(y) & a & 1 & 1 & 1
\end{array}
\]
is a pair of residuated mappings between $\mathbf P$ and $\mathbf Q$.
\end{example}

The next lemma provides a method for checking if a given residuated mapping is range-closed respectively dually range-closed.

\begin{lemma}
Let $\mathbf P=(P,\le,1)$ be a poset with $1$, $\mathbf Q=(Q,\le,0)$ a poset with $0$ and $(f,f^*)$ a pair of residuated mappings between $\mathbf P$ and $\mathbf Q$. Then
\begin{enumerate}[{\rm(i)}]
\item $f$ is range-closed if and only if $f\bigg(P\setminus\Big(\downarrow\big(f(1)\big)\Big)\bigg)=(0,f(1)]$.
\item $f$ is dually range-closed if and only if $f^*\bigg(Q\setminus\Big(\uparrow\big(f(1)\big)\Big)\bigg)=[f^*(0),1)$.
\end{enumerate}
\end{lemma}

\begin{proof}
$\text{}$
\begin{enumerate}[(i)]
\item $f$ is range-closed if and only if $f(P)=[0,f(1)]$. According to Lemma~\ref{lem5}, $f^{-1}(0)= \\
\downarrow\big(f^*(0)\big)$.
\item $f$ is dually range-closed if and only if $f^*(Q)=[f^*(0),1]$. According to Lemma~\ref{lem5}, $(f^*)^{-1}(1)=\uparrow\big(f(1)\big)$.
\end{enumerate}
\end{proof}








Authors' address:

Ivan Chajda \\
Palack\'y University Olomouc \\
Faculty of Science \\
Department of Algebra and Geometry \\
17.\ listopadu 12 \\
771 46 Olomouc \\
Czech Republic \\
ivan.chajda@upol.cz

Helmut L\"anger \\
TU Wien \\
Faculty of Mathematics and Geoinformation \\
Institute of Discrete Mathematics and Geometry \\
Wiedner Hauptstra\ss e 8-10 \\
1040 Vienna \\
Austria, and \\
Palack\'y University Olomouc \\
Faculty of Science \\
Department of Algebra and Geometry \\
17.\ listopadu 12 \\
771 46 Olomouc \\
Czech Republic \\
helmut.laenger@tuwien.ac.at
\end{document}